\documentclass[11pt]{article}
\usepackage{latexsym,amsfonts,amssymb,amsmath,amsthm}
\usepackage{graphicx}
\usepackage{cite}
\usepackage[usenames,dvipsnames]{color}
\usepackage{ulem}
\usepackage{amsmath,amscd}
\usepackage[colorlinks=true]{hyperref}
\hypersetup{urlcolor=blue, citecolor=red}

\usepackage{hyperref}
\hypersetup{
    colorlinks=true,
    linkcolor=red,
    filecolor=magenta,
    urlcolor=cyan,
    pdftitle={Sharelatex Example},
    bookmarks=true,
    pdfpagemode=FullScreen,}

\begin{document}
\setlength{\baselineskip}{16pt}

\parindent 0.5cm
\evensidemargin 0cm \oddsidemargin 0cm \topmargin 0cm \textheight
22cm \textwidth 16cm \footskip 2cm \headsep 0cm

\newtheorem{theorem}{Theorem}[section]
\newtheorem{lemma}{Lemma}[section]
\newtheorem{proposition}{Proposition}[section]
\newtheorem{definition}{Definition}[section]
\newtheorem{example}{Example}[section]

\newtheorem{remark}{Remark}[section]
\newtheorem{corollary}{Corollary}[section]
\newtheorem{property}{Property}[section]
\numberwithin{equation}{section}
\newtheorem{mainthm}{Theorem}
\newtheorem{mainlem}{Lemma}

\numberwithin{equation}{section}

\def\p{\partial}
\def\I{\textit}
\def\R{\mathbb R}
\def\C{\mathbb C}
\def\u{\underline}
\def\l{\lambda}
\def\a{\alpha}
\def\O{\Omega}
\def\e{\epsilon}
\def\ls{\lambda^*}
\def\D{\displaystyle}
\def\wyx{ \frac{w(y,t)}{w(x,t)}}
\def\imp{\Rightarrow}
\def\tE{\tilde E}
\def\tX{\tilde X}
\def\tH{\tilde H}
\def\tu{\tilde u}
\def\d{\mathcal D}
\def\aa{\mathcal A}
\def\DH{\mathcal D(\tH)}
\def\bE{\bar E}
\def\bH{\bar H}
\def\M{\mathcal M}
\renewcommand{\labelenumi}{(\arabic{enumi})}

\def\disp{\displaystyle}
\def\undertex#1{$\underline{\hbox{#1}}$}
\def\card{\mathop{\hbox{card}}}
\def\sgn{\mathop{\hbox{sgn}}}
\def\exp{\mathop{\hbox{exp}}}
\def\OFP{(\Omega,{\cal F},\PP)}
\newcommand\JM{Mierczy\'nski}
\newcommand\RR{\ensuremath{\mathbb{R}}}
\newcommand\CC{\ensuremath{\mathbb{C}}}
\newcommand\QQ{\ensuremath{\mathbb{Q}}}
\newcommand\ZZ{\ensuremath{\mathbb{Z}}}
\newcommand\NN{\ensuremath{\mathbb{N}}}
\newcommand\PP{\ensuremath{\mathbb{P}}}
\newcommand\abs[1]{\ensuremath{\lvert#1\rvert}}

\newcommand\normf[1]{\ensuremath{\lVert#1\rVert_{f}}}
\newcommand\normfRb[1]{\ensuremath{\lVert#1\rVert_{f,R_b}}}
\newcommand\normfRbone[1]{\ensuremath{\lVert#1\rVert_{f, R_{b_1}}}}
\newcommand\normfRbtwo[1]{\ensuremath{\lVert#1\rVert_{f,R_{b_2}}}}
\newcommand\normtwo[1]{\ensuremath{\lVert#1\rVert_{2}}}
\newcommand\norminfty[1]{\ensuremath{\lVert#1\rVert_{\infty}}}

\title{The Morse Smale property for time-periodic scalar reaction-diffusion equation on the circle}

\author {
\\
Tingting Su
  $\,$ and Dun Zhou
  \thanks{Corresponding author. Partially supported by NSF of China No.11971232, 12071217} \\
 School of Mathematics and Statistics,
 \\ Nanjing University of Science and Technology
 \\ Nanjing, Jiangsu, 210094, P. R. China
  \\
}
\date{}

\maketitle

\begin{abstract}
We study the Morse-Smale property for the following scalar semilinear parabolic equation on the circle $S^1$,
\begin{equation*}
u_{t}=u_{xx}+f(t,u,u_{x}),\,\,t>0,\,x\in S^{1}=\mathbb{R}/2\pi \mathbb{Z},
\end{equation*}
where $f$ is a $C^2$ function and $T$-periodic in $t$.  Assume that the equation admits a compact global attractor $\mathcal{A}$ and let $P$ be the Poincar\'{e} map of this equation. We exclude homoclinic connection for hyperbolic fixed points of $P$ and prove that stable and unstable manifolds for any two heteroclinic hyperbolic fixed points of $P$ intersect transversely. Further, this equation admits the Morse-Smale property provided that all $\omega$-limit sets (in the case $f(t,u,u_x)=f(t,u,-u_x)$, the $\omega$-limit set is just a fixed point) of the corresponding Poincar\'{e} map are hyperbolic.
\end{abstract}

\section{Introduction}

The study of the structural stability of differential equations can be traced back to the work of Poincar\'{e}  on the three-body problem in celestial mechanics and was first appeared in Andronov and Pontrjagin \cite{37And}. Later, the structural stability was extensively studied in the frame of vector fields on compact smooth manifolds of finite dimensions. A system or equation is structurally stable means that the qualitative behavior of the trajectories can persistent by small perturbations(to be exact $C^1$-small perturbations). In 1960(see \cite{60Sma}), Smale proposed the concept of Morse-Smale system: (1) the non-wandering set consists only in a finite number of hyperbolic equilibria and hyperbolic orbits; (2) the stable and unstable manifolds of all the connecting hyperbolic fixed points and periodic orbits are all intersect transversely. The Morse-Smale system on any finite dimensional compact manifold is structurally stable was proved by Palis and Smale \cite{Palis1968,Palis-Smale,Smale}. Particularly, for a 2-dimensional compact manifold, Morse-Smale property is established for generic $C^1$-vector fields; under some dissipative assumptions, the generic Morse-Smale property is also correct for vector fields on $\mathbb{R}^2$. Nevertheless, Morse-Smale vector fields no longer dense for a general 3-dimensional system. In fact, the existence of transverse homoclinic connecting hyperbolic orbits lead to chaotic behaviors(see \cite{Smale-65}).

As for infinite systems that generated by partial differential equations, results related to structural stability and local stability are more recent. Similar to the finite dimensional case of vector fields on compact manifold or on $\mathbb{R}^n$, one can also define Morse-Smale property for infinite systems. It is proved that Morse-Smale systems generated by disspative parabolic equations or more generally by disspative in finite time smoothing equations are sturctural stable
(see \cite{Oliva2000,84Hale}). To our best knowledge, the first relevant result on the structural stability of parabolic equations should belong to Henry \cite{85Hen} for the following equation
\[
u_t=u_{xx}+f(x,u,u_x),\quad (t,x)\in(0,\infty)\times (0,1),
\]
with Dirichlet, Neumann or Robin boundary conditions. It was proved in \cite{85Hen}(see also Angenent\cite{86Ang}) that the stable and unstable manifolds of two connecting hyperbolic fixed points of the above equation intersect transversely; and this result combine with already proved gradient structure in \cite{Zelen} imply that the generic Morse-Smale property of such equation. After then, Chen, et al. in \cite{92Chen} obtained the Morse-Smale property for $f=f(t,x,u,u_x)$ of $t$ is periodic under the assumption that all the fixed points of the corresponding Poincar\'{e} map are hyperbolic. It is worth pointing out that scalar semilinear parabolic equation defined on a bounded domain $\Omega\subset \mathbb{R}^n$, $n\geq 2$, is no longer gradient in general. Nonetheless, for special nonlinear term $f=f(x,u)$, the system is also a gradient system, and hence, Morse-Smale property for some special equations on bounded domain $\Omega\subset \mathbb{R}^n$, $n\geq 2$ established(see \cite{BP97}). For general $f=f(x,u,\nabla u)$, the generic tranversality of heteroclinic and homoclinic orbits was proved in the recent work of Brunovsk\'{y}, et al in \cite{BJR}.

As it mentioned before, Morse-Smale system is an important class of structural stable system. In spite of this fact, there are not many results that can be used to verify if Morse-Smale properties holds for a given dynamical system, especially for that high-dimensional or infinite-dimensional systems. In this situation, it is necessary and important to select some typical systems to study their Morse-Smale properties. Here, we focus on the following typical infinite-dimensional system
\begin{equation}\label{r-d-eq}
\begin{cases}
u_{t}=u_{xx}+f(t,u,u_{x}),\quad t>0,x\in S^{1}=\mathbb{R}/2\pi \mathbb{Z},\\
u(0,x)=u_{0},
\end{cases}
\end{equation}
where $f$ is a $C^2$ function with $f(t+T,u,u_{x})=f(t,u,u_{x})$,  $T>0.$  Our goal is to understand the global stability of the above equation \eqref{r-d-eq}, and obtain it's Morse-Smale property.

Although the above system is not a gradient system, by virtue of the powerful tool of zero number function (see  \cite{88AngFie,Massatt1986,88Mat}), the flow of \eqref{r-d-eq} still has particular properties similar to one-dimensional or two-dimensional systems(depending on the form of nonlinear term $f$). For instance, in autonomous case, if $f=f(x,u,u_x)$ is general dependent, the celebrated Poincar\'{e}-Bendixson type theorem was obtained in  Fiedler and Mallet-Paret \cite{Fiedler}, that is, any $\omega$-limit set of  \eqref{r-d-eq} is either a single periodic orbit or it consists of equilibria and connecting (homoclinic and heteroclinic) orbits. Based on this conclusion, transversality of the stable and unstable manifolds of hyperbolic equilibria and periodic orbits, and generic hyperbolicity were obtained in \cite{08Cza,10Jol362}. After then, Joly and Raugel \cite{10Jol27} established the generic Morse-Smale property for the system with a compact global attractor. For the special case, $f=f(u,u_x)$, any $C^1$-bounded solution of \eqref{r-d-eq} approaches a set of functions of the form
\[
\Sigma \phi=\{\phi(\cdot+a)\,|\, a\in S^1\},
\]
where $\phi$ satisfies the following equation
\begin{equation*}
  \phi_{xx}+c\phi_x+f(\phi,\phi_x)=0,\quad x\in S^1,
\end{equation*}
for some $c\in \mathbb{R}$(see \cite{88AngFie,Massatt1986, 88Mat}). In other words, any $C^1$-bounded solution of \eqref{r-d-eq} approximates either a constant equilibrium (i.e., $\phi$ is constant and $c=0$), a rotating wave (i.e., $\phi$ is nonconstant and $c\neq 0$) or a one-dimensional manifold of standing waves (i.e., $\phi$ is nonconstant and $c=0$). Particularly, if $f(u,u_x)=f(u,-u_x)$, then any bounded solution is asymptotic to an equilibrium (see \cite{88Mat}).

As for \eqref{r-d-eq} is a time periodic equation, that is, $f(t+T,u,u_{x})=f(t,u,u_{x})$ for some $T>0$,  Sandstede and Fiedler \cite{92San} proved  that any  $\omega$-limit set generated by bounded solution of \eqref{r-d-eq} can be viewed as a subset of the two-dimensional torus $\mathcal{T}^1\times S^1$ carrying a linear flow. Particularly, if $f(t,u,u_x)=f(t,u,-u_x)$, then any $\omega$-limit set of bounded solution of \eqref{r-d-eq} is a $T$-periodic orbit (see Chen and Matano \cite{89Che}). For $f=f(t,u,u_x)$ of \eqref{r-d-eq} is almost periodic in $t$, the long time behavior of bounded solutions of \eqref{r-d-eq} are clearly characterized, see the series work by Shen, et al \cite{16She,19Shen,20She,22She}.

There are three reasons driving us to consider the Morse-Smale property of $T$-periodic equation \eqref{r-d-eq}: there are few results discussing the Morse-Smale properties for periodic systems; another is the delicate example constructed in \cite{92San} shown that any time-periodic planar vector field can be embedded into \eqref{r-d-eq} with certain nonlinearity $f=f(t,x,u,u_x)$ (see also the comments in \cite{Hale}), the appearance of chaotic behaviors exclude the possibility of structural stable of periodic system \eqref{r-d-eq} with general $f=f(t,x,u,u_x)$;  finally  the recent work by Shen, et al \cite{21She} displays that any non-wandering point of \eqref{r-d-eq} is a limit point, which pave the way for proving the Morse-Smale property of \eqref{r-d-eq}.

Assume that $f\in C^2(\mathbb{R}\times\mathbb{R}\times\mathbb{R},\mathbb{R})$ and $X^{\alpha}$($\frac{3}{4}<\alpha<1$) is the fractional power space associated with the operator $u\rightarrow -u_{xx}:H^{2}(S^{1})\rightarrow L^{2}(S^{1})$, denoted by $A$, then $X^{\alpha}$ is compactly embedded in $C^{1}(S^{1})$. By the standard theory for parabolic equations (see \cite{81Hen}), for any $u_0\in X^{\alpha}$, \eqref{r-d-eq} defines (locally) a unique solution $u(t,\cdot;u_0)$ in $X^{\alpha}$ with $u(0,x;u_0)=u_0(x)$. It then follows from \cite{81Hen} and the standard priori estimates for parabolic equations, if $u(t,\cdot;u_0)$ is bounded in $X^{\alpha}$ in the existence interval of the solution, then $u$ is a globally defined classical solution(see Section \ref{global-attractor} for more detailed discussions).

 Let $P$ be the Poincar\'{e} map \eqref{r-d-eq}, we have the following

 \begin{itemize}

 \item[$\bullet$] (see {\bf Theorem} \ref{index-cor}) {\it There is no homoclinic connecting hyperbolic fixed point of $P$.}

 \item[$\bullet$] (see {\bf Theorem} \ref{tranver-theor}) {\it  Assume that $\varphi_{+}, \varphi_{-}$ are hyperbolic fixed points of $P$, then the stable manifold of $\varphi_{+}$ and unstable manifold of $\varphi_{-}$ intersect transversely in $X^{\alpha}$.}

 \item[$\bullet$] (see {\bf Theorem} \ref{morse-smale-theor}) {\it Assume moreover \eqref{r-d-eq} admits a compact global attractor and all the $\omega$-limit sets of $P$ are hyperbolic, then the discrete semiflow $\{P^{n}, n\geq0\}$ is Morse-Smale.}
\end{itemize}

We have some comments in the following.

Firstly, our result is not a simple extension for autonomous systems in \cite{08Cza, 10Jol27} to periodic systems. In fact, the main difficulty to prove the transversality of invariant manifolds between two hyperbolic fixed points is to construct suitable spaces such that the spaces can span the full space. In autonomous case, for hyperbolic periodic orbit, the whole orbit is both belong to the local stable manifolds and local unstable manifolds of itself, which is very important for the proving of transversality for connecting hyperbolic equilibrium elements(hyperbolic equilibrium or hyperbolic periodic orbit). While in our periodic case, this property no longer holds true, as we need to consider the transversality on Poincar\'{e} section. The proof of Theorem \ref{tranver-theor} depends heavily on the recent progress for hyperbolic $\omega$-limit set in  \cite[Theorem 2.1]{92San} and \cite[Lemma 3.2, Theorem 5.1]{19Shen}(see also Remark \ref{hyp-eigenspace}).

Secondly, although the methods in \cite{89Che} deed provide some help for our proof, our result still cannot be viewed as a simple generalization of results in \cite{89Che} for \eqref{r-d-eq} with separated boundary conditions to periodic boundary condition. As the corresponding linearized systems enjoy different properties, for instance, the Floquet multipliers and their corresponding eigenspaces for the current system are much more complicated than that in \cite{89Che}.

Thirdly, the recent characterization of non-wandering points in \cite{21She} makes it easier to verify the conditions for Morse-Smale property.

Fourth, the zero number function plays an essential role in characterization the global stability and longtime behavior of \eqref{r-d-eq}. Similar functions  widely exist in many other systems, see \cite{DLZ,Ma-Sm,MS,Smillie1984,VMV}. From the view of dynamical systems, such function can help us to reduce a high-dimensional or infinite dimensional system to lower or finite dimensional system. Here, the zero number function is always used throughout the whole paper.

The paper is organized as follows. In Section 2.1, we impose some suitable conditions on \eqref{r-d-eq}, so that the equation admits local unique classical solution in a suitable space for a given initial data and admits a compact global attractor; then in Section 2.2, we list some properties of zero number function, and its relationship between Floquet eigenspaces; in Section 2.3, we introduce some concepts of discrete Morse-Smale system. In Section 3, we study the long time behavior of solutions of linear aysmptotic periodic parabolic equations on the circle. In Section 4, we present our main results and give detailed proof, we first exclude the existence of homoclinic connecting hyperbolic periodic orbits; and then prove the transversality of stable and unstable manifolds between heteroclinic connecting hyperbolic periodic orbits; finally obtain the Morse-Smale property under the assumption that all  $\omega$-limit sets of $P$ in the non-wandering set are hyperbolic. In the Appendix, we list some important results from \cite{92Chen} that often used in Section 3 and Section 4.

\section{Preliminaries}
In this section, we introduce some notions and concepts which will be often used later.
\subsection{Existence of global attractor}\label{global-attractor}

To carry out our research work, we assume that \eqref{r-d-eq} is dissipative, that is, the system \eqref{r-d-eq} admits a global attractor $\mathcal{A}$. For this purpose, we make some restrictions on $f$ in equation\eqref{r-d-eq}:
\begin{itemize}
  \item [(a.1)]$f:\mathbb{R}\times \mathbb{R} \times \mathbb{R}\rightarrow \mathbb{R}$ is of class $C^{2},$
  \item [(a.2)]there exist $0\leq\gamma<2$ and a continuous function $\eta:[0,\infty)\rightarrow[0,\infty)$ such that

\begin{equation}\label{assum-f}
|f(t,y,z)|\leq \eta(r)(1+|z|^{\gamma}),\quad (t,y,z)\in \mathbb{R}\times[-r,r]\times \mathbb{R} \quad\mbox{for each}\quad r>0,
\end{equation}

\begin{equation}\label{assum-yf}
yf(t,y,0)<0,\quad (t,y)\in \mathbb{R} \times \mathbb{R},\quad |y|\geq \delta,\quad\mbox{for some}\quad \delta>0.
\end{equation}
\end{itemize}

The above conditions can guarantee that the system has a global attractor in a suitable space(see \cite[Section 2]{10Jol27}). For the sake of completeness, we give a detailed explanation in the following.

Define the operator $A:H^{2}(S^{1})\subset L^{2}(S^{1})\rightarrow L^{2}(S^{1}),$
$$Au=-u_{xx},\quad u\in H^{2}(S^{1}).$$
Since $A$ is a non-negative definite self-adjoint operator, it is a non-negative sectorial operator. Hence, we define fractional power spaces
$$ X^{\alpha}=D(A^{\alpha}),\quad\alpha\geq0,$$
with norm $||u||_{X^{\alpha}}=||A^{\alpha}u||_{L^{2}(S^{1})}+||u||_{L^{2}(S^{1})}, u\in X^{\alpha}$.
Note that $X^{0}=L^{2}(S^{1}),X^{1}=H^{2}(S^{1}),$
$$X^{\alpha}=[L^{2}(S^{1}),H^{2}(S^{1})]_{\alpha}=H^{2\alpha}(S^{1}),\quad\alpha\in (0,1).$$
Since $H^{2}(S^{1})$ is compactly embedded into $L^{2}(S^{1})$ , $A$ has compact resolvent.

The equation \eqref{r-d-eq} can be rewritten into the following  abstract Cauchy problem in $L^{2}(S^{1})$,
\begin{align}\label{abstract-Cauchy-eq}
\left\{
\begin{aligned}
&u_{t}+Au=F(t,u),\\
&u(0)=u_{0},
\end{aligned}
\right.
\end{align}
where $F(t,u)(x)=f(t,u(x),u_{x}(x)), x\in S^{1}$. For a fixed $\alpha\in(\frac{3}{4},1)$, by virtue of $f\in C^{2}$, $F$ is locally H\"{o}lder continuous in $t$ and locally Lipschitz continuous in $u$ on open subset $U\subset\mathbb{R}^{+}\times X^{\alpha}$,  and for every bounded set $B\subset U$, the image $F(B)$ is bounded in  $X^{0}$. Therefore, by \cite[Theorem 3.3.3, Theorem 3.3.4]{81Hen}, for each $u_{0}\in U$, there exists a unique local forward $X^{\alpha}$ solution(classical solution) defined on a maximal interval $[0,\tau_{u_{0}})$, where $\tau_{u_{0}}=+\infty $, or
$$
\tau_{u_{0}}<+\infty \quad\mbox{and} \quad \limsup_{t\rightarrow\tau_{u_{0}}}||u(t,u_{0})||_{ X^{\alpha}}=+\infty.
$$

By using \eqref{assum-yf} and the maximum principle for parabolic equations, we know that if $||u_{0}||_{L^{\infty}(S^{1})}\leq \delta+R$ for some $R\geq0$, then there exists a positive constant $\zeta(\delta,R)$ such that
\begin{equation}\label{u-norm-control}
||u(t,u_{0})||_{L^{\infty}(S^{1})}\leq \delta+Re^{-\zeta t},\quad t\in[0,\tau_{u_{0}}).
\end{equation}
For simplicity, we assume that $1<\gamma<2$ in \eqref{assum-f}, as a matter of fact, if there exists $0<\gamma<1$ such that $$|f(t,y,z)|\leq \eta(r)(1+|z|^{\gamma}),$$
 then by using Young's inequality one has
$$
|z|^{\gamma}\leq\frac{\gamma}{\gamma+1}|z|^{\gamma+1}+\frac{1}{\gamma+1},
$$
that is, there exists a continuous function $\eta^{'}$
such that
$$
|f(t,y,z)|\leq \eta^{'}(r)(1+|z|^{\gamma+1}).
$$
Based on \eqref{assum-f} and \eqref{u-norm-control}, we obtain
\begin{equation}\label{F-norm-control}
\begin{aligned}
||F(t,u(t,u_{0}))||_{X^{0}}&=(\int_{S^{1}} |f(t,u,u_{x})|^{2}dx)^{\frac{1}{2}}\\
&\leq c(||u(t,u_{0})||_{L^{\infty}(S^{1})})(1+||u(t,u_{0})||^{\gamma}_{W^{1,2\gamma}}),\quad t\in (0,\tau_{u_{0}}).
\end{aligned}
\end{equation}
We now estimate $||u(t,u_{0})||_{W^{1,2\gamma}}$, let $r$ large enough and  $\max\{\frac{\gamma}{2},\alpha\}<\beta<1$ such that
$$
\max\{\frac{r-2\gamma}{\gamma(r-2)}, \frac{1}{2\beta}\}<\theta<\frac{1}{\gamma},
$$
the following inequality is satisfied
\begin{equation}
\begin{aligned}
\|u(t,u_{0})\|_{W^{1,2\gamma}(S^{1})}\leq c_{\theta}\|u(t,u_{0})\|^{\theta}_{H^{2\beta}(S^{1})}\|u(t,u_{0})\|^{1-\theta}_{L^{r}(S^{1})},\quad t\in(0,\tau_{u_{0}}), \nonumber
\end{aligned}
\end{equation}
(see \cite[Proposition 1.2.2]{00Cho}).
Given that $L^{\infty}(S^{1})$ is continuously embedded in $L^{r}(S^{1}),$
\begin{equation}\label{u-norm-control-2}
\begin{aligned}
\|u(t,u_{0})\|_{W^{1,2\gamma}(S^{1})}\leq \tilde{c_{\theta}}\|u(t,u_{0})\|^{\theta}_{H^{2\beta}(S^{1})}\|u(t,u_{0})\|^{1-\theta}_{L^{\infty}(S^{1})},\quad t\in(0,\tau_{u_{0}}),
\end{aligned}
\end{equation}
by virtue of \eqref{u-norm-control}, \eqref{F-norm-control} and \eqref{u-norm-control-2}
\begin{equation}
\begin{aligned}
||F(t,u(t,u_{0}))||_{X^{0}}\leq \tilde{c}(||u(t,u_{0})||_{L^{\infty}(S^{1})})(1+||u(t,u_{0})||^{\theta\gamma}_{X^{\beta}}),\quad t\in(0,\tau_{u_{0}}), \nonumber
\end{aligned}
\end{equation}
with $\theta\gamma<1$.
It then follows from \cite[Corollary 3.3.5]{81Hen} that the forward $X^{\beta}$ solution of \eqref{abstract-Cauchy-eq} exists globally in time, i.e. $\tau_{u_{0}}=\infty.$

We now define the Poincar\'{e} map $P:\Psi\rightarrow X^{\alpha}$ as the following
$$P(u_{0})=u(T,x;u_{0}),\quad u_{0}\in \Psi,$$
in which $\Psi=\{u_{0}\in X^{\alpha}|\tau_{u_{0}}>T\}$. Let $\Psi_{0}=X^{\alpha},\Psi_{1}=\Psi,\Psi_{n+1}=P^{-1}(\Psi_{n}),n\geq1,$
where $P^{-1}(\Psi_{n})$ is the preimage of $\Psi_{n}$ under the map $P$, define $P^{n}:\Psi_{n}\rightarrow X^{\alpha}$, where $P^{-n}(S)$ is the preimage of $S\subset X^{\alpha}$ under the map $P^{n}$. The solution $u(nT,x;u_{0})=P^{n}(u_{0})$ of equation \eqref{r-d-eq} generates a discrete semiflow $\{P^{n}:\Psi_{n}\rightarrow X^{\alpha}, n\geq0\}$ in $X^{\beta}$.

Noticing that there exist local forward  $X^{\alpha}$ solution and  $X^{\alpha}$ enter $X^{1}$, for $t>0$, we may consider $X^{\alpha}$ global solution as $X^{\beta}$ global solution.

Recall that \eqref{u-norm-control} indicates
$$\limsup_{t\rightarrow+\infty}||u(t,u_{0})||_{L^{\infty}(S^{1})}\leq \delta,$$
then by \cite[Theorem 4.1.1]{00Cho}, there exists a constant $K_{1}$, such that
$$\limsup_{t\rightarrow+\infty}||u(t,u_{0})||_{X^{\alpha}}\leq K_{1}.$$ Therefore, the discrete semiflow $\{P^{n},n\geq0\}$ is point dissipative in $X^{\alpha}$. Recall that $A$ has compact resolvent, by\cite[Theorem 3.3.1]{00Cho}, $P^{n}$ is a compact map in $X^{\alpha}$, and hence the discrete semiflow $\{P^{n},n\geq0\}$ has a global attractor $\mathcal{A}$ in $X^{\alpha}$.

\subsection{Zero number function and Floquet theory for fixed point}
Given a $C^{1}$ function $\varphi:S^{1}\rightarrow \mathbb{R},$ the zero number of $\varphi$ is defined
$$z(\varphi)=card\{x\in S^{1}:\varphi(x)=0\}.$$
According to \cite{88Ang,82Mat,98Chen}, we have the following lemma.
\begin{lemma}\label{zero-number-lem}
Let $v(t,x)$ be a nontrivial classical solution of the linear parabolic equation
\begin{equation}
\left\{
\begin{aligned}
&v_{t}=v_{xx}+a(t,x)v_{x}+b(t,x)v,\quad x\in S^{1},\\
&v(0)=v_{0}, \nonumber
\end{aligned}
\right.
\end{equation}
where $a,b$ are bounded continuous functions, then the zero number of $v(t)$ has the following properties:\\
(i)$z(v(t))<\infty$ for any $t>0$,\\
(ii)if $0\leq t_{1}\leq t_{2}$, then $z(v(t_{1}))\geq z(v(t_{2}))$,\\
(iii) $z(v(t))$ can drops only finite many times, and there exists $T_{0}>0$ such that $v(t)$ has only simple zeros in $S^{1}$ as $t\geq T_{0}$.
\end{lemma}

Let $\mathcal{L}(X^{\alpha})$ be the bounded linear operators from $X^{\alpha}$ into $X^{\alpha}$, $\sigma(l)$ be the spectrum of operator $l\in \mathcal{L}(X^{\alpha})$.  For $\varphi\in X^{\alpha}$, if $P(\varphi)=\varphi$, we say that $\varphi$ is a fixed point of Poincar\'{e} map $P$, the Morse index $\operatorname{ind}(\varphi)$ of fixed point $\varphi$ is denoted as the number of elements of $\sigma(DP(\varphi))$ whose moduli are greater than one, in which $DP(\varphi)$ is the Fr\'{e}chet derivative of $P$ at $\varphi$. If $\sigma(DP(\varphi))$ dose not intersect the unit circle, we say that $\varphi$ is a hyperbolic fixed point of $P$. Hereafter, we always assume $\operatorname{ind}(\varphi)>0$.

Suppose $\varphi$ is a fixed point of $P$, $u(t,x;\varphi)$ is the solution of equation \eqref{r-d-eq}, then the linearized equation about $u(t,x;\varphi)$ is
\begin{subequations}\label{fixed-point-eq}
\begin{align}
&v_{t}=v_{xx}+\hat{c}(t,x)v_{x}+\hat{d}(t,x)v,\\
&v(0)=v_{0},
\end{align}
\end{subequations}
where $$\hat{c}(t,x)=\partial f_{3}(t,u(t,x;\varphi),u_{x}(t,x;\varphi)),$$
$$\hat{d}(t,x)=\partial f_{2}(t,u(t,x;\varphi),u_{x}(t,x;\varphi)),$$
and $\hat{S}:X^{\alpha}\rightarrow X^{\alpha}$ is the evolution operator, then $v(t,\cdot;v_{0})=\hat{S}(t,0)v_{0}$ is a solution of \eqref{fixed-point-eq} with $v(0,\cdot;v_{0})=v_0$;  moreover, $\hat{S}((n+1)T,nT)=\hat{S}(T,0), n\geq1,$ for simplicity, let $\hat{S}\triangleq \hat{S}(T,0)$.

Then, one has the following characterization about eigenvalues and eigenspaces of $\hat{S}$ from \cite[Theorems 2.1, 2.2]{88AngFie}.
\begin{lemma}\label{eigen-value}
Let  $\{\lambda_{0}, \lambda_{j}, \tilde{\lambda}_{j}\}_{j\geq1}$ be eigenvalues of $\hat{S}$, then\\
(i) $$
|\lambda_{0}|>|\lambda_{1}|\geq\tilde{|\lambda_{1}|}>\cdots>|\lambda_{m}|\geq|\tilde{\lambda}_{m}|>|\lambda_{m+1}|\geq|\tilde{\lambda}_{m+1}|>\cdots
$$
(ii) let $E_{0}$ be the eigenspace of $\lambda_{0}$, and $E_{j}$ be the eigenspace of $\{\lambda_{j},\tilde{\lambda}_{j}\}, j\geq1$, then, for each $ \phi\in E_{j}\setminus\{0\},  j\geq0, \phi$  has only  simple zeros, and $z(\phi)=2j.$
\end{lemma}

\begin{remark}\label{hyp-eigenspace}
Assume that $\varphi$ is a hyperbolic fixed point of $P$, it then follows by \cite[Theorem 2.1]{92San} and \cite[Lemma 3.2, Theorem 5.1]{19Shen} that $\varphi$ is spatially-homogeneous and if the dimension of unstable manifold $\varphi$  is not equal zero then it must be odd, and
$$
|\lambda_{0}|>|\lambda_{1}|=|\tilde{\lambda}_{1}|>|\lambda_{2}|=|\tilde{\lambda}_{2}|>\ldots>|\lambda_{j}|=|\tilde{\lambda}_{j}|>\cdots
$$
\end{remark}

\subsection{Morse-Smale System}
In this subsection, we introduce some concepts related to the Morse-Smale system. Let  $\varphi$ be a hyperbolic fixed point of $P$, then the stable and unstable manifolds of $\varphi$ are defined as follows,
\begin{equation}
W^{s}(\varphi)=\{\psi\in X^{\alpha}|P^{n}(\psi)\rightarrow\varphi \quad\mbox{in} \quad X^{\alpha} \quad \mbox{as}\quad n\rightarrow\infty\}, \nonumber
\end{equation}
\begin{equation}
W^{u}(\varphi)=\{\psi\in X^{\alpha}|P^{-n}(\psi)\rightarrow\varphi \quad \mbox{in}\quad X^{\alpha} \quad\mbox{as}\quad n\rightarrow\infty\}. \nonumber
\end{equation}

Let $u(t,x;u_{0})$ be a solution of equation \eqref{r-d-eq}, then the linearized equation about $u(t,x;u_{0})$ is the following,
\begin{subequations}\label{tranver-line-eq}
\begin{align}
&v_{t}=v_{xx}+\bar{c}(t,x)v_{x}+\bar{d}(t,x)v,\\
&v(0)=v_{0},
\end{align}
\end{subequations}
where
$$\bar{c}(t,x)=\partial f_{3}(t,u(t,x;u_{0}),u_{x}(t,x;u_{0})),$$
$$\bar{d}(t,x)=\partial f_{2}(t,u(t,x;u_{0}),u_{x}(t,x;u_{0})).$$
Let  $S(t,0;u_{0}):X^{\alpha}\rightarrow X^{\alpha}$ be the solution operator of \eqref{tranver-line-eq}, then for each $v_0\in X^{\alpha}$, $v(t,\cdot;v_{0})=S(t,0;u_{0})v_{0}$ is the solution of \eqref{tranver-line-eq} with $v(0,\cdot;v_{0})=v_0$. In the following, we give an equivalent characterization about tangent spaces of stable and unstable manifolds for hyperbolic fixed points.
\begin{lemma}\label{def-Tu0}
Assume $\varphi$ is a hyperbolic fixed point of $P$,  then $W^{s}(\varphi)$ and $W^{u}(\varphi)$ are $C^1$ immersed submanifolds of $X^{\alpha}$; moreover,
\begin{itemize}
 \item[\rm{ (i)}] for any $u_{0}\in W^{s}(\varphi)$, the tangent space of $ W^{s}(\varphi)$ at $u_0$ is given by
\begin{equation}
T_{u_{0}}W^{s}(\varphi)=\{v_{0}\in X^{\alpha}|\limsup_{n\rightarrow\infty}\|S(nT,0;u_{0})v_{0}\|^{\frac{1}{n}}_{X^{\alpha}}<1\}; \nonumber
\end{equation}

\item[\rm{(ii)}]for any $u_{0}\in W^{u}(\varphi)$, the tangent space of $ W^{u}(\varphi)$ at $u_0$ is given by
\begin{equation}
\begin{aligned}
T_{u_{0}}W^{u}(\varphi)=\{v_{0}\in X^{\alpha}|&v_{0}\quad\mbox{has a backward continuation for \eqref{tranver-line-eq} and}\\ &\limsup_{n\rightarrow\infty}\|S(-nT,0;u_{0})v_{0}\|^{\frac{1}{n}}_{X^{\alpha}}<1\}. \nonumber
\end{aligned}
\end{equation}
\end{itemize}
\end{lemma}

Note that the Fr\'{e}chet derivative
$DP(u_{0})=S(T,0;u_{0})$, the proof of Lemma \ref{def-Tu0} is similar to that of \cite[Lemma 4.2]{92Chen}. To prove Lemma \ref{def-Tu0}, by using Lemma \ref{unstable-manif} and Lemma \ref{stable-manif}, we only need to check the following.
\begin{lemma}\label{injective-map}
The map $P|_{\mathcal{A}}:\mathcal{A}\rightarrow\mathcal{A}$ and $DP(u_{0}):X^{\alpha}\rightarrow X^{\alpha} (u_{0}\in \mathcal{A})$ are injective.
\end{lemma}
\begin{proof}
Let $u(t,x;u_{0}), u(t,x;u_{1})$ be two solutions of equation \eqref{abstract-Cauchy-eq} with $u_{0}, u_{1}\in \mathcal{A}$. Suppose $P(u_{0})=P(u_{1})$, that is, $u(T,x;u_{0})=u(T,x;u_{1})$. Let $v(t)=u(t,x;u_{0})-u(t,x;u_{1})$, then $v(t)$ satisfies the following equation
\begin{equation}
\left\{
\begin{aligned}
&v_{t}+Av=G(t,u)(x),\quad t>0, x\in S^{1},\\
&v(0)=u_{0}-u_{1}, \nonumber
\end{aligned}
\right.
\end{equation}
where $G(t,u)(x)=F(t,u(t,x;u_{0}))-F(t,u(t,x;u_{1}))=c_{1}(t,x)v_{x}+d_{1}(t,x)v,$ $$c_{1}(t,x)=\int_{0}^{1}\partial f_{3}(t,su(t,x;u_{0})+(1-s)u(t,x;u_{1}),su_{x}(t,x;u_{0})+(1-s)u_{x}(t,x;u_{1}))ds,
$$
$$
d_{1}(t,x)=\int_{0}^{1}\partial f_{2}(t,su(t,x;u_{0})+(1-s)u(t,x;u_{1}),su_{x}(t,x;u_{0})+(1-s)u_{x}(t,x;u_{1}))ds
$$
and $v(T)=0$.

Note that $A: X^{1}\rightarrow X^{0}$ is a sectorial operator, $X^{\frac{1}{2}}=D(A^{\frac{1}{2}})=[X^{1}, X^{0}]_{\frac{1}{2}}$. Furthermore, $G(t,u)(x)$ is locally H\"{o}lder continuous in $t$ and locally Lipschitz continuous in $u$,  then by the standard semigroup theory (see \cite[Definition 3.3.1, Theorem 3.3.4, Corollary 3.3.5 and Theorem 3.5.2]{81Hen}) one has
$$v\in C([0,\infty),X^{\alpha})\cap C^{1}((0,\infty),X^{0})\cap C((0,\infty),X^{1}),$$
and
\begin{equation}
\begin{aligned}
\|G(t,u)\|_{X^{0}}&=(\int_{S^{1}}|c_{1}(t,x)v_{x}+d_{1}(t,x)v|^{2}dx)^{\frac{1}{2}},\\
&\leq L\parallel v\parallel_{X^{\frac{1}{2}}},\quad t\in[0,\infty), \nonumber
\end{aligned}
\end{equation}
where $L$ is a constant.

It then follows from \cite[Proposition 7.1.1]{00Cho} that $v(t)=0, t\in [0,T]$, a contradiction, $P|_{\mathcal{A}}$ is injective.

We turn to prove that $DP(u_{0})$ is also injective. Recall that $DP(u_{0})=S(T,0;u_{0})$, to prove $DP(u_{0})$ is injective it is sufficient to prove that if $S(T,0;u_{0})w_0=0$, then $w_0=0$. For this purpose, for each $w_{0}\in \mathcal{A}$, let $w(t,x;w_{0})=S(t,0;u_{0})w_0$,
then it satisfies the following equation
\begin{equation}
\left\{
\begin{aligned}
&w_{t}+Aw=D_{u}F(t,u(t,x;u_{0}))w,\quad t>0,x\in S^{1},\\
&w(0)=w_{0}, \nonumber
\end{aligned}
\right.
\end{equation}
this is equivalent to
\begin{equation}\label{w0-direction-eq}
\left\{
\begin{aligned}
&w_{t}=w_{xx}+c_{2}(t,x)w_{x}+d_{2}(t,x)w,\quad t>0, x\in S^{1},\\
&w(0)=w_{0}, \nonumber
\end{aligned}
\right.
\end{equation}
where $$c_{2}(t,x)=\partial f_{3}(t,u(t,x;u_{0}),u_{x}(t,x;u_{0})),$$
$$d_{2}(t,x)=\partial  f_{2}(t,u(t,x;u_{0}),u_{x}(t,x;u_{0})).$$

Suppose that $DP(u_{0})w_{0}=0,$ then  $w(T,x;w_{0})=0$, and hence
\begin{equation}
\begin{aligned}
&\|D_{u}F(t,u(t,\cdot;u_{0}))w\|_{X^{0}}=\|\partial f_{2}(t,u,u_{x})w+\partial f_{3}(t,u,u_{x})w_{x}\|_{X^{0}}\\
&\leq C_{1}\|w(t)\|_{X^{0}}+C_{2}\|w(t)\|_{X^{\frac{1}{2}}}\\
&\leq M\|w(t)\|_{X^{\frac{1}{2}}}, \nonumber
\end{aligned}
\end{equation}
in which $C_{1},C_{2}$ depend on
$$\sup_{(t,x)\in [0,T]\times S^{1}}|\partial f_{2}(t,u,u_{x})|,$$
$$\sup_{(t,x)\in [0,T]\times S^{1}}|\partial f_{3}(t,u,u_{x})|,$$
respectively.

Again by \cite[Proposition7.1.1]{00Cho}, we have
$$w(t)=0,\quad t\in[0,T],$$
particularly, $ w_{0}=0$, $DP(u_{0})$ is injective.
\end{proof}

Assume $\varphi_{+}, \varphi_{-}$ be two hyperbolic fixed points of $P$. If the tangent space of a point $u_0$ in the intersection of stable manifold of $\varphi_{+}$ and unstable manifold of $\varphi_{-}$ can generate full space, we say that the two manifolds intersect transversely at $u_0$, that is
\[
T_{u_0}W^s(\varphi_{+})+T_{u_0}W^u(\varphi_{-})=X^{\alpha}.
\]

A point $\psi\in X^{\alpha}$ is said to be a {\it non-wandering point} of $P$ if $\psi\in\bigcap_{n=0}^{\infty}\Psi_{n}$ and for any neighborhood $U$ of $\psi$ and any nonnegative integer $N$ there exist a $\phi\in U$ and $n\geq N$ such that $P^{n}(\phi)\in U$.

The discrete dynamical system $P$ is called {\it Morse-Smale} if it possesses a global attractor $\mathcal{A}$ with the following\\
(i) $P|_{\mathcal{A}}:\mathcal{A}\rightarrow\mathcal{A}$ and $DP(\varphi):X^{\alpha}\rightarrow X^{\alpha}(\varphi\in \mathcal{A})$ are injective;\\
(ii) the non-wandering set $\mathcal{N}$ of $P|_{\mathcal{A}}$ consists of a finite number of hyperbolic periodic points of $P$;\\
(iii) for any  pair of periodic points $\varphi_{+}$ and $\varphi_{-}$, the local stable manifold of $\varphi_{+}$ and the global unstable manifold of $\varphi_{-}$ intersect transversely in $X^{\alpha}$(see \cite{84Hale}).

\begin{remark}
 From  \cite[Theorem 2.1]{92San} and \cite[Theorem 5.1]{19Shen} we know that  $\varphi$ is a hyperbolic fixed point of Poincar\'{e} map $P$ of \eqref{r-d-eq}, if and only if it is a hyperbolic periodic point of $P$.
\end{remark}

\section{Linear asymptotic periodic systems}
In this section, we investigate some properties of linear asymptotic periodic equation of \eqref{r-d-eq}.

Consider the following two linear equations
\begin{subequations}\label{asy-equa}
\begin{align}
&v_{t}=v_{xx}+c(t,x)v_{x}+d(t,x)v,\quad (t,x)\in \mathbb{R}^{+}\times S^{1},\\
&v(0)=v_{0},
\end{align}
\end{subequations}
and
\begin{subequations}\label{lin-periodic-eq}
\begin{align}
&v^{p}_{t}=v^{p}_{xx}+\tilde{c}(t,x)v^{p}_{x}+\tilde{d} (t,x)v^{p},\quad (t,x)\in \mathbb{R}^{+}\times S^{1},\\
&v^{p}(0)=v^{p}_{0},
\end{align}
\end{subequations}
where
$$\tilde{c}(t+T,x)= \tilde{c} (t,x),\quad \forall t>0,\quad x\in S^{1},$$
$$\tilde{d}(t+T,x)=\tilde{d}(t,x),\quad \forall t>0,\quad x\in S^{1},$$
 for some $T>0$, $c(t,x), d(t,x), \tilde{c}(t,x),\tilde{d}(t,x)\in C(\mathbb{R}^{+}\times S^{1}),$ and all are locally H\"{o}lder continuous in $t$.

Assume that
$$
\lim_{n\rightarrow\pm\infty}\|c-\tilde{c}\|_{L^{\infty}([nT,(n+1)T]\times S^{1})}+\|d-\tilde{d}\|_{L^{\infty}([nT,(n+1)T]\times S^{1})}=0,
$$
 then \eqref{lin-periodic-eq} is called the asymptotic periodic equation of \eqref{asy-equa}.
Let $S(t,0):X^{\alpha}\rightarrow X^{\alpha}$ and $S_{p}(t,0):X^{\alpha}\rightarrow X^{\alpha}$ be solution operators of \eqref{asy-equa} and \eqref{lin-periodic-eq}, then $v(t,\cdot;v_{0})=S(t,0)v_{0}$, $v^{p}(t,\cdot;v^{p}_{0})=S_{p}(t,0)v^{p}_{0}$ are solutions of  \eqref{asy-equa} and \eqref{lin-periodic-eq} respectively. For simplicity, write $S_{p}=S_{p}(T,0)=S_{p}((n+1)T,nT).$

The following lemma tells us that $S_{p}$ and $S((n+1)T,nT)$ can be arbitrarily close as $n\to \pm\infty$.

\begin{lemma}\label{close-two-operators-lem}
$S_{p}$ has the following relationship with $S((n+1)T,nT)$ in $X^{\alpha}$,\\
(i)$\lim_{n\rightarrow\infty}\|S((n+1)T,nT)-S_{p}\|_{\mathcal{L}(X^{\alpha})}=0$,\\
(ii)$\lim_{n\rightarrow-\infty}\|S((n+1)T,nT)-S_{p}\|_{\mathcal{L}(X^{\alpha})}=0$.
\end{lemma}
\begin{proof}
Consider the following abstract equation
\begin{equation}\label{asy-equa-semigroup}
\left\{
\begin{aligned}
&v^{n}_{s}+Av^{n}=h^{n}(s)v^{n},\quad 0< s\leq T,\quad x\in S^{1},\\
&v^{n}(0)=v^{n}_{0},
\end{aligned}
\right.
\end{equation}
and
\begin{equation}\label{lin-periodic-eq-semigroup}
\left\{
\begin{aligned}
&v^{p}_{s}+Av^{p}=h^{p}(s)v^{p},\quad 0< s\leq T,\quad x\in S^{1},\\
&v^{p}(0)=v^{p}_{0},
\end{aligned}
\right.
\end{equation}
where $Au=-u_{xx}$, $h^{n}(s)v^{n}=c(nT+s,x)v^{n}_{x}+d(nT+s,x)v^{n},$ $h^{p}(s)=\tilde{c}(s,x)v^{p}_{x}+\tilde{d}(s,x)v^{p}$. It is not hard to see that $h^{n}(s),h^{p}(s)\in\mathcal{L}(X^{\alpha},X^{0})$ are H\"{o}lder continuous in $s$. By standard theory for analytic semigroups(see \cite[Theorem 3.3.3]{81Hen}), equations \eqref{asy-equa-semigroup} and \eqref{lin-periodic-eq-semigroup} admit classical solutions in $X^{\alpha}$.

Let $v^{n}_{0}=v^{p}_{0}=\psi\in X^{\alpha}$, then $v^{n}(s,\cdot;\psi)=S(nT+s,nT)\psi$, $v^{p}(s,\cdot;\psi)=S_{p}(s,0)\psi$ respectively. Note that $L^{\infty}(S^{1})\hookrightarrow X^{0},$ and $X^{\alpha}\hookrightarrow C^1(S^1)$, there exists $C_{\alpha}>0$ such that
\begin{equation*}
\begin{aligned}
&\sup_{s\in(0,T]}\|h^{n}(s)-h^{p}(s)\|_{\mathcal{L}(X^{\alpha},X^{0})}\\
&=\sup_{s\in(0,T]}\frac{\|c(nT+s,\cdot)\psi_{x}+d(nT+s,\cdot)\psi-\tilde{c}(s,\cdot)\psi_{x}-\tilde{d}(s,\cdot)\psi\|_{X^{0}}}{\|\psi\|_{X^{\alpha}}}\\
&\leq\sup_{s\in(0,T]}\frac{\|c(nT+s,\cdot)-\tilde{c}(s,\cdot)\|_{L^{\infty}(S^1)}\|\psi_{x}\|_{X^{0}}+\|d(nT+s,\cdot)-\tilde{d}(s,\cdot)\|_{L^{\infty}(S^1)}\|\psi\|_{X^{0}}}{\|\psi\|_{X^{\alpha}}}\\
&\leq C_{\alpha}\sup_{s\in(0,T]}[\|c(nT+s,\cdot)-\tilde{c}(s,\cdot)\|_{L^{\infty}(S^1)}+\|d(nT+s,\cdot)-\tilde{d}(s,\cdot)\|_{L^{\infty}(S^1)}].\\
\end{aligned}
\end{equation*}
By asymptotic periodicity, one has
\begin{equation}\label{close-two-operators}
\sup_{s\in[0,T]}\|h^{n}(s)-h^{p}(s)\|_{\mathcal{L}(X^{\alpha},X^{0})}\to 0, \quad n\to \infty.
\end{equation}
By standard theory of analytic semigroups, we know that $$\|A^{\alpha}e^{-At}\|\leq M_{1}t^{-\alpha},$$
where $M_{1}$ is a positive constant depending on $A$ and $\alpha$. Let $M_p=\sup_{0\leq s\leq T}\|h^p(s)\|_{\mathcal{L}(X^{\alpha},X^{0})}$, then
$$
 ||v^{p}(s,\cdot;\psi)||_{X^{\alpha}}\leq M_{2}\|\psi\|_{X^{\alpha}}+\int_0^{s}\frac{M_{3}} {(s-\tau)^{\alpha} }\|v^p(\tau,\cdot;\psi)\|_{X^{\alpha}}d\tau,
$$
where $M_2$ depends on $A$, $M_3$ depends on $A, \alpha, M_p$. It then follows from the Gronwall's inequality (see Exercise 4* in \cite[Chapter 7]{81Hen}) that

\begin{equation}
 \parallel v^{p}(s,\cdot;\psi)\parallel_{X^{\alpha}}\leq M_{4}\parallel\psi\parallel_{X^{\alpha}},\quad \forall s\in [0,T], \nonumber
\end{equation}
where $M_{4}>0$ is a constant depending on $ T , A, \alpha, M_p$. By the variation of constants formula, we have
\begin{equation}
\begin{aligned}
v^{n}(s,x;\psi)-v^{p}(s,x;\psi)=\int_{0}^{s}&e^{-A(s-\xi)}\{(h^{n}(\xi)-h^{p}(\xi))v^{p}(\xi,x;\psi)\\
&+h^{n}(\xi)(v^{n}(\xi,x;\psi)-v^{p}(\xi,x;\psi))\}d\xi.\nonumber
\end{aligned}
\end{equation}
In view of  \eqref{close-two-operators}, for the given $0<\epsilon \ll 1$, we may assume that
\[
\begin{split}
 \sup_{s\in [0,T]}\|h^n(s)\|_{\mathcal{L}(X^{\alpha},X^{0})}\leq M_p+1, \quad n\gg 1,\\
 \sup_{s\in[0,T]}\|h^{n}(s)-h^{p}(s)\|_{\mathcal{L}(X^{\alpha},X^{0})}<\epsilon, \quad n\gg 1.
\end{split}
\]
Therefore,
\begin{equation}
\begin{aligned}
&\| v^{n}(s,\cdot;\psi)-v^{p}(s,\cdot;\psi)\|_{X^{\alpha}}\\
&\leq\int_{0}^{s}\frac{M_1}{(s-\xi)^{\alpha}}\{\parallel(h^{n}(\xi,\cdot)-h^{p}(\xi,\cdot))\parallel _{\mathcal{L}(X^{\alpha},X^{0})}\parallel v^{p}(\xi,\cdot;\psi)\parallel_{X^{\alpha}}\\
&+\parallel h^{n}(\xi,\cdot)\parallel_{\mathcal{L}(X^{\alpha},X^{0})}\parallel(v^{n}(\xi,\cdot;\psi)-v^{p}(\xi,\cdot;\psi))\parallel_{X^{\alpha}}\}d\xi\\
&\leq \frac{M_1M_4s^{1-\alpha}}{1-\alpha}\varepsilon \parallel\psi\parallel_{X^{\alpha}}+{ M_{1} }(M_{p}+1)\int_{0}^{s}(s-\xi)^{-\alpha}\parallel v^{n}(\xi,\cdot;\psi)-v^{p}(\xi,\cdot;\psi)\parallel_{X^{\alpha}}d\xi\nonumber, \quad n\gg 1.
\end{aligned}
\end{equation}
Again by the Gronwall's inequality,
$$
\parallel v^{n}(s,\cdot;\psi)-v^{p}(s,\cdot;\psi)\parallel_{X^{\alpha}}
\leq M_{5}\varepsilon\parallel\psi\parallel_{X^{\alpha}},\quad n\gg 1,
$$
where $M_{5}$ is a constant depending on $\alpha, T, M_{1}, M_{4}, M_{p}$.

Thus
\begin{equation}
\lim_{n\rightarrow\infty}\|S((n+1)T,nT)-S_{p}\|_{\mathcal{L}(X^{\alpha})}=\lim_{n\rightarrow\infty}\sup_{\|\psi\|_{X^{\alpha}}=1} \|v^{n}(s,\cdot;\psi)-v^{p}(s,\cdot;\psi)\|_{X^{\alpha}}
=0. \nonumber
\end{equation}
Similarly, replacing $n\rightarrow\infty$ with $n\rightarrow-\infty$, we can obtain
\begin{equation}
\lim_{n\rightarrow-\infty}\|S((n+1)T,nT)-S_{p}\|_{\mathcal{L}(X^{\alpha})}=0. \nonumber
\end{equation}
\end{proof}

From Lemma \ref{eigen-value}, we know that the eigenvalues of $S_{p}$ can be denoted as $\{\lambda_{0},\lambda_{j},\tilde{\lambda}_{j}\}_{j\geq1},$ where
$|\lambda_{0}|>|\lambda_{j}|\geq|\tilde{\lambda}_{j}|>|\lambda_{j+1}|\geq\tilde{|\lambda}_{j+1}|,j\geq1$,  and $\lim_{j\rightarrow\infty}|\lambda_{j}|=0$, $\lim_{j\rightarrow\infty}|\tilde{\lambda}_{j}|=0.$

Let $\Lambda$ be the set of all nonnegative numbers $r$ that satisfy
$$\sigma(S_{p})\cap\{z\in \mathbb{C}:|z|=r\}\neq\emptyset\,$$
that is $\Lambda=\{r_{0}, r_{j}, \tilde{r}_{j}, j\geq1\}$, where $r_{0}=|\lambda_{0}|,r_{j}=|\lambda_{j}|, \tilde{r}_{j}=|\tilde{\lambda}_{j}|, j\geq1.$

For any $\psi\in X^{\alpha},$ and for any integer $m\geq0$, define
$$\rho_{\infty}(m,\psi)=\limsup_{n\rightarrow\infty}\|v(nT,\cdot;mT,\psi)\|_{X^{\alpha}}^{\frac{1}{n}},$$
where $v(t,x;mT,\psi)$ is the solution of equation \eqref{asy-equa} with $v(mT,\cdot;mT,\psi)=\psi$.

For integer $k\geq0$, define
\begin{equation}\label{def-Fk+}
F_{k}^{+}(m)=\{\psi\in X^{\alpha}|\rho_{\infty}(m,\psi)\leq r_{k}\}.
\end{equation}
Then we have the following.

\begin{proposition}\label{pro-Fk+}
For each integer $m\geq0$, we have
\begin{itemize}
  \item[\rm{(i)}]$X^{\alpha}=F_{0}^{+}(m)\supset F_{1}^{+}(m)\supset\ldots,\quad \mbox{and}\quad \bigcap_{k\geq0} F_{k}^{+}(m)=\{0\}$;

 \item[\rm{(ii)}] if $\psi\in F_{k}^{+}(m)\backslash F_{k+1}^{+}(m)$, then
$$\lim_{n\rightarrow\infty}\|v(nT,\cdot;mT,\psi)\|_{X^{\alpha}}^{\frac{1}{n}}=r_{k},$$
or $$\lim_{n\rightarrow\infty}\|v(nT,\cdot;mT,\psi)\|_{X^{\alpha}}^{\frac{1}{n}}=\tilde{r}_{k},$$
and there exists a subsequence$\{n_{k}T\}\subset\{nT\}_{n\geq0},$ such that
$$\frac{v(n_{k}T,x;mT,\psi)}{\|v(n_{k}T,\cdot;mT,\psi)\|_{X^{\alpha}}}$$
converges into the eigenspace $E_{k}$  as $n_{k}\rightarrow\infty$ ($E_{k}$ be eigenspaces of $S_{p}$ as defined in Lemma \ref{eigen-value});

\item[\rm{(iii)}] $z(\psi)\geq2k, \forall\psi\in F_{k}^{+}(m)\setminus\{0\}.$
\end{itemize}
\end{proposition}
\begin{proof}
(i) It is not hard to see that $F_{0}^{+}(m)\supset F_{1}^{+}(m)\supset\cdots$. We first prove $X^{\alpha}=F_{0}^{+}(m)$.

By Lemma \ref{close-two-operators-lem}(i), we know that for any $\psi\in X^{\alpha}$
\begin{equation}
\begin{aligned}
v((n+1)T,x;mT,\psi)&=S((n+1)T,nT)v(nT,x;mT,\psi)\\
&=S_{p}v(nT,x;mT,\psi)+R_{n}v(nT,x;mT,\psi),\nonumber
\end{aligned}
\end{equation}
where $R_{n}=S((n+1)T,nT)-S_{p}\in \mathcal{L}(X^{\alpha})$ with $\lim_{n\rightarrow\infty}\|R_{n}\|_{X^{\alpha}}=0$. It then follows from Lemma \ref{norm-spectrum} that there exists an eigenvalue $\lambda$ of $S_{p}$ such that
$$\lim_{n\rightarrow\infty}\|v(nT,\cdot;mT,\psi)\|^{\frac{1}{n}}_{X^{\alpha}}=|\lambda|,$$
which proved $X^{\alpha}=F_{0}^{+}$.

Since  $\rho_{\infty}(m,\psi)=0$ is only established for $\psi=0$ and $\lim_{n\rightarrow\infty}r_{n}=0 (resp. \lim_{n\rightarrow\infty}\tilde{r}_{n}=0)$,  we have
$$\bigcap_{k=0}^{\infty}F_{k}^{+}(m)=\{0\}.$$

(ii)Note that
$$
F_{k}^{+}(m)\backslash F_{k+1}^{+}(m)=\{\psi\in X^{\alpha}|r_{k+1}<\rho_{\infty}(m,\psi)\leq r_{k}\},
$$
for any given $\psi\in F_{k}^{+}(m)\backslash F_{k+1}^{+}(m)$, one has
$$\lim_{n\rightarrow\infty}\|v(nT,\cdot;mT,\psi)\|_{X^{\alpha}}^{\frac{1}{n}}=r_{k},$$
or
$$\lim_{n\rightarrow\infty}\|v(nT,\cdot;mT,\psi)\|_{X^{\alpha}}^{\frac{1}{n}}=\tilde{r}_{k}.$$
Choose $0<\delta<r_{k}$ (resp. $0<\delta<\tilde{r}_{k}$) be such that $\sigma(S_{p})\bigcap\{{z\in \mathbb{C}: } r_{k}-\delta\leq|z|\leq r_{k}+\delta\}=\{r_{k}\}$ (resp. $\sigma(S_{p})\bigcap\{z\in \mathbb{C}: \tilde{r}_{k}-\delta\leq|z|\leq \tilde{r}_{k}+\delta\}=\{\tilde{r}_{k}\}$) and let
$$\mathcal{X}(\alpha)=\lim_{n\rightarrow\infty}\frac{\|P(\alpha)v(nT)\|}{\|Q(\alpha)v(nT)\|},\quad \alpha\in[0,\infty)\backslash\Lambda,$$
where $P(\alpha)$ and $Q(\alpha)$ are the projection operators about the decomposition of the spectrum $\sigma(S_{p})\bigcap\{z\in \mathbb{C}:|z|>\alpha\}$ and $\sigma(S_{p})\bigcap\{z\in \mathbb{C}:|z|<\alpha\}$ respectively. Then by Lemma \ref{projection-lem}
$$
\mathcal{X}(r_{k}-\delta)=\infty, \mathcal{X}(r_{k}+\delta)=0,
$$
$$
(\mbox{resp.}\quad\mathcal{X}(\tilde{r}_{k}-\delta)=\infty, \mathcal{X}(\tilde{r}_{k}+\delta)=0).
$$
Let $P_{*}=P(r_{k}-\delta)-P(r_{k}+\delta)$ (resp. $P_{*}=P(\tilde{r}_{k}-\delta)-P(\tilde{r}_{k}+\delta)$),  it is obvious that $\dim P_{*}<\infty$, and according to Lemma \ref{conver-posi-seque}, there exists a subsequence $\{n_{k}T\}\subset\{nT\}_{n\geq0}$ such that $\frac{v(n_{k}T,x;mT,\psi)}{\|v(n_{k}T,\cdot;mT,\psi)\|_{X^{\alpha}}}$ converges into the eigenspace $E_{k}$ as $n_{k}\rightarrow\infty$.

(iii) Let $\hat{v}(n_{k}T)=\frac{v(n_{k}T,x;mT,\psi)}{\|v(n_{k}T,\cdot;mT,\psi)\|_{X^{\alpha}}},$
by virtue of (ii), we assume that $\hat{v}(n_{k}T)\rightarrow \phi_{k}$, as $n_{k}\rightarrow\infty$, where $\phi_{k}\in P_{*}X^{\alpha}$, it then follows by Lemma \ref{close-two-operators-lem}(i) that
$$
S((n_{k}+1)T,n_{k}T)\hat{v}(n_{k}T)\rightarrow S_{p}\phi_{k}=\lambda_{k}\phi_{k} \quad \mbox{as}\quad n_{k}\rightarrow\infty,
$$
then by Lemma \ref{zero-number-lem} (iii)
$$z(S((n_{k}+1)T,n_{k}T)\hat{v}(n_{k}T))=z(\lambda_{k}\phi_{k})=2k \quad\mbox{as}\quad n_{k}\gg 1,$$
and according to Lemma \ref{zero-number-lem} (ii), we can obtain
$$
z(\psi)\geq z(\hat{v}((n_{k}+1)T))=2k,\quad \psi\in F_{k}^{+}(m).
$$
\end{proof}

For each non-positive integer $m$ let
$$F_{\infty}^{-}(m)=\{\psi\in X^{\alpha}| v(t,x;mT,\psi)\mbox{is the backward solution of} \eqref{asy-equa}\}.$$
And for each $\psi\in F_{\infty}^{-}(m)$, define
$$
\rho_{-\infty}(m,\psi)=\liminf_{n\rightarrow-\infty}\|v(nT,\cdot;mT,\psi)\|_{X^{\alpha}}^{\frac{1}{n}},
$$
for each integer $k\geq0$, define
\begin{equation}\label{def-Fk-}
F_{k}^{-}(m)=\{\psi\in F_{\infty}^{-}(m)|\rho_{-\infty}(m,\psi)>r_{k}\}.
\end{equation}

\begin{proposition}\label{pro-Fk-}
For each integer $m\leq0,$ we have
\begin{itemize}
  \item [\rm{(i)}]$\{0\}=F_{0}^{-}(m)\subset F_{1}^{-}(m)\subset\cdots;$
  \item [\rm{(ii)}]if $\psi\in F_{k}^{-}(m)\backslash \{0\}$, then there exists integer $0\leq l<k,$ such that
$$\lim_{n\rightarrow-\infty}\|v(nT,\cdot;mT,\psi)\|_{X^{\alpha}}^{\frac{1}{n}}=r_{l},$$
or
$$\lim_{n\rightarrow-\infty}\|v(nT,\cdot;mT,\psi)\|_{X^{\alpha}}^{\frac{1}{n}}=\tilde{r}_{l},$$
and there exists a subsequence $\{n_{k}T\}\subset\{nT\}_{n\leq0},$ such that
$$\frac{v(n_{k}T,x;mT,\psi)}{\|v(n_{k}T,\cdot;mT,\psi)\|_{X^{\alpha}}}$$
converges into the projection space $E_{l}$ as $n_{k}\rightarrow-\infty $ ($E_{l}$ be eigenspaces of $S_{p}$ as defined in Lemma \ref{eigen-value});
\item [\rm{(iii)}]$z(\psi)<2k, \forall\psi\in F_{k}^{-}(m).$
\end{itemize}
\end{proposition}

\begin{proof}
The proof is similar to that for Proposition \ref{pro-Fk+}, we omit it here.
\end{proof}

\begin{lemma}\label{Fk-iso-Ek}
Assume that $E_{j}, j\geq0 $ be eigenspaces of $S_{p}$ as defined in Lemma \ref{eigen-value}, then for sufficiently large $m\in \mathbb{N}$, one has
\begin{itemize}
  \item [\rm{(i)}]$F_{k}^{+}(m)$ is isomorphic to $E_{k}\oplus E_{k+1}\oplus\cdots;$
  \item [\rm{(ii)}]$F_{k}^{-}(m)$ is isomorphic to $E_{0}\oplus E_{1}\oplus\cdots\oplus E_{k-1}$.

\end{itemize}
\end{lemma}
\begin{proof}
(i) It is not hard to see that there exists $\varepsilon>0$ such that $(r_{k}+\varepsilon, \tilde{r}_{k-1}-\varepsilon)\cap\Lambda=\emptyset.$
Let $\lambda\in (r_{k}+\varepsilon, \tilde{r}_{k-1}-\varepsilon)$, define
 $$\delta(\lambda,R)=\sum_{k=1}^{n}\lambda^{k-n-1}\|V^{n-k}QR_{k-1}\|+\sum_{k=n+1}^{\infty}\lambda^{k-n-1}\|U^{n-k}PR_{k-1}\|,$$
 with $P=P(\lambda), Q=Q(\lambda), R_{k}(k\geq0)$ being same to Proposition \ref{pro-Fk+}, and $U=S_{p}|_{P(\lambda)X^{\alpha}}, V=S_{p}|_{Q(\lambda)X^{\alpha}}$, by virtue of Lemma \ref{posi-norm-control}, we have $\delta(\lambda, R)<1$.

From the definition of $F_{k}^{+}(m)$ we know that
$$F_{k}^{+}(m)=\{\psi\in X^{\alpha}|\sup_{n\geq0}\lambda^{-n}\|v(nT,\cdot;mT,\psi)\|_{X^{\alpha}}<\infty\},$$
it then follows from Lemma \ref{posi-isomorphism} that $F_{k}^{+}(m)$ is isomorphic to $E_{k}\oplus E_{k+1}\oplus\cdots$.

(ii) Similarly, by Lemma \ref{nega-isomorphism} and Lemma \ref{nega-norm-control}, we can obtain that $F_{k}^{-}(m)$ is isomorphic to $E_{0}\oplus E_{1}\oplus E_{k-1}.$
\end{proof}

\section{Main Results and Proofs}
Assume that $f$ in \eqref{r-d-eq} is a $C^2$ function and $P$ is the corresponding Poincar\'{e} map, then our main results in this section are the following.

\begin{theorem}\label{index-cor}
Assume that $\varphi_{+}, \varphi_{-}$ be two hyperbolic fixed points of $P$, and if
$$
W^{s}(\varphi_{+})\cap W^{u}(\varphi_{-})\backslash\{\varphi_{+},\varphi_{-}\}\neq \emptyset,
$$
then $\operatorname{ind}(\varphi_{-})>\operatorname{ind}(\varphi_{+}).$
\end{theorem}

\begin{theorem}\label{tranver-theor}
Assume that $\varphi_{+}, \varphi_{-}$ are hyperbolic fixed points of $P$, then $W^{s}(\varphi_{+})$ and $W^{u}(\varphi_{-})$ intersect transversely in $X^{\alpha}$.
\end{theorem}

\begin{theorem}\label{morse-smale-theor}
Assume moreover $f$ satisfies (a.1)-(a.2) in equation \eqref{r-d-eq}, and if all the $\omega$-limit sets of $P$ are hyperbolic, then the discrete semiflow $\{P^{n}, n\geq0\}$ is Morse-Smale.
\end{theorem}
\begin{remark}
{\rm
  Theorem \ref{index-cor} excludes the existence of homoclinic connection between hyperbolic fixed points. Note that in the case $f(t,u,u_x)=f(t,u,-u_x)$,  the non-wandering set consists of fixed points of $P$(see \cite[Theorem 4.1]{21She}), the assumption in Theorem \ref{morse-smale-theor} is equivalent to that all the fixed points are hyperbolic.}
\end{remark}

\subsection{Proof of Theorem \ref{index-cor}}
To prove Theorem \ref{index-cor}, we first prove the following.
\begin{lemma}\label{non-homoclinin}
Assume $\varphi_{+}$ is a hyperbolic fixed point of the map $P$, then for any $u_{0}\in W^{s}(\varphi_{+})\backslash\{\varphi_{+}\}$ $($resp. $u_{0}\in W^{u}(\varphi_{-})\backslash\{\varphi_{-}\})$, one has $z(P(u_{0})-u_{0})>\operatorname{ind}(\varphi_{+})$. $($resp. $ z(P(u_{0})-u_{0})< \operatorname{ind}(\varphi_{-})).$
\end{lemma}
\begin{proof}
Given $u_{0}\in W^{s}(\varphi_{+})\backslash\{\varphi_{+}\}$, let $u(t,x;u_{0})$ and $u(t,x;\varphi_{+})$ be two solutions of equation \eqref{r-d-eq}. It then follows from the definition of $W^{s}(\varphi_{+})$ that $u(t,x;u_{0})-u(t,x;\varphi_{+})\to 0$ in $X^{\alpha}$, as $t\rightarrow\infty$.

Let $v(t,x;v_{0})=u(t+T,x;u_{0})-u(t,x;u_{0})$, then $v$ satisfies
\begin{subequations}\label{no-homo-eq}
\begin{align}
&v_{t}=v_{xx}+c(t,x)v_{x}+d(t,x)v,\\
&v(0)=v_{0},
\end{align}
\end{subequations}
where $v_{0}=P(u_{0})-u_{0}$, and
$$c(t,x)=\int_{0}^{1}\partial f_{3}(t,su(t+T,x;u_{0})+(1-s)u(t,x;u_{0}),su_{x}(t+T,x;u_{0})+(1-s)u_{x}(t,x;u_{0}))ds,$$
$$d(t,x)=\int_{0}^{1}\partial f_{2}(t,su(t+T,x;u_{0})+(1-s)u(t,x;u_{0}),su_{x}(t+T,x;u_{0})+(1-s)u_{x}(t,x;u_{0}))ds.$$
By virtue of the smoothness of $f$, we know that $c(t,x), d(t,x)\in C(\mathbb{R}^+\times S^1)$ and be locally H\"{o}lder continuous in $t$; and moreover, by the asymptotic periodic of $u(t,x;u_{0})$
$$
c(t,x)\rightarrow \partial f_{3}(t,u(t,x;\varphi_{+}),u_{x}(t,x;\varphi_{+})),
$$
$$d(t,x)\rightarrow \partial f_{2}(t,u(t,x;\varphi_{+}),u_{x}(t,x;\varphi_{+}))
,$$
as $t\to \infty$.

Note that $\varphi_{+}$ is hyperbolic, by Remark \ref{hyp-eigenspace},
$$
|\lambda_{0}|>|\lambda_{1}|=|\tilde{\lambda}_{1}|>|\lambda_{2}|=|\tilde{\lambda}_{2}|>\ldots>|\lambda_{j}|=|\tilde{\lambda}_{j}|>\cdots
$$
and $\tilde r_j=r_j$, $j=1,\cdots.$ Hence, there exists $m^+\in \mathbb{N}$ such that $2m^++1=ind(\varphi_{+})$ and $r_{m^+}>1>r_{m^++1}$.

Let $F_{k}^{+}$ and $F_{k}^{-}$ be subspaces of $X^{\alpha}$ associated with \eqref{no-homo-eq}(as defined in \eqref{def-Fk+} and \eqref{def-Fk-} respectively). We {\it claim } that $v_{0}\in F_{m^++1}^{+}(0)$. In fact, if $v_{0}\notin  F_{m^++1}^{+}(0)$, then $\rho_{\infty}(0,v_{0})>r_{m^++1}$, that is $$\limsup_{n\rightarrow\infty}\|v(nT,\cdot;v_{0})\|^{\frac{1}{n}}_{X^{\alpha}}>r_{m^++1},$$
this combines with Lemma \ref{norm-spectrum} lead to
$$\limsup_{n\rightarrow\infty}\|v(nT,\cdot;v_{0})\|^{\frac{1}{n}}_{X^{\alpha}}=r\geq r_{m^+}>1,$$
a contradiction to that $\|v(nT,\cdot;v_{0})\|_{X^{\alpha}}\rightarrow 0$, as $n\rightarrow\infty$, we proved the {\it claim }. It then follows by Proposition \ref{pro-Fk+}(iii) that $z(v_{0})=z(P(u_{0})-u_{0})\geq2(m^++1)>\operatorname{ind} (\varphi^+)$.

For $u_{0}\in W^{u}(\varphi_{-})\backslash\{\varphi_{-}\}$, $u(t,x;u_{0})$ and $u(t,x;\varphi_{-})$ be two backward solutions of equation \eqref{r-d-eq}. By the definition of $W^{u}(\varphi_{-})$, we have $u(t,x;u_{0})-u(t,x;\varphi_{-})\to 0$ in $X^{\alpha}$ as $t\rightarrow-\infty$.
Then $v(t,x;v_{0})=u(t+T,x,u_{0})-u(t,x;u_{0})$ is the backward solution of equation \eqref{no-homo-eq}.

Let $m^-\in\mathbb{N}$ be such that $2m^-+1=ind(\varphi_{-})$, hence $r_{m^-}>1>r_{m^-+1}$, and $v_{0}\in F_{m^-+1}^{-}(0)$. In fact, if $v_{0}\notin  F_{m^-+1}^{-}(0)$, then $\rho_{-\infty}(0,v_{0})\leq r_{m^-+1}$, that is
$$
\liminf_{n\rightarrow\infty}\frac{1}{\|v(-nT,\cdot;v_{0})\|^{\frac{1}{n}}_{X^{\alpha}}}\leq r_{m^-+1},
$$
and there exists a subsequence $\{n_{k}T\}\subset\{nT\}_{n\geq0}$ such that $$
\frac{1}{\|v(-n_{k}T,\cdot;v_{0})\|^{\frac{1}{n_k}}_{X^{\alpha}}}\leq \frac{r_{m^-+1}+1}{2},\quad k\gg 1.
$$
As a consequence, $\|v(-n_{k}T,\cdot;v_{0})\|^{\frac{1}{n_k}}_{X^{\alpha}}\geq \frac{2}{r_{m^-+1}+1}>1$, $k\gg 1$, this leads to a contradiction to that $\|v(nT,\cdot;v_{0})\|_{X^{\alpha}}\\ \rightarrow 0$ as $n\rightarrow-\infty$. Thus, $v_{0}\in F_{m^-+1}^{-}(0)$. It then follows from Proposition \eqref{pro-Fk-}(iii) that $z(v_{0})=z(P(u_{0})-u_{0})\leq 2m^-<\operatorname{ind} (\varphi_-)$.
\end{proof}

\begin{proof}[Proof of Theorem \ref{index-cor}]
  It is a direct result of Lemma \ref{non-homoclinin}.
\end{proof}
\subsection{Proof of Theorem \ref{tranver-theor}}
\begin{proof}[Proof of Theorem \ref{tranver-theor}]
It is sufficient to check that
$$T_{u_{0}}W^{s}(\varphi_{+})+T_{u_{0}}W^{u}(\varphi_{-})=X^{\alpha}$$
for  any $u_{0}\in W^{s}(\varphi_{+})\cap W^{u}(\varphi_{-})\backslash\{\varphi_{+},\varphi_{-}\}$.

Note that $\varphi_{+},\varphi_{-}$ are hyperbolic fixed points of $P$,  by Remark \ref{hyp-eigenspace}, we assume that $2m^{+}+1=\operatorname{ind}(\varphi_{+}), 2m^{-}+1=\operatorname{ind}(\varphi_{-})$. It then follows from Lemma \ref{def-Tu0}(i) that
$$T_{u_{0}}W^{s}(\varphi_{+})=F_{m^{+}+1}^{+}(0)=\{v_{0}\in X^{\alpha}|\limsup_{n\rightarrow\infty}\|v(nT,\cdot;v_{0})\|^{\frac{1}{n}}_{X^{\alpha}}\leq r_{m^{+}+1}\},$$
where $F_{m^{+}+1}^{+}(0)$ is as defined in \eqref{def-Fk+} and $v(t,x;v_{0})$ is the solution of \eqref{tranver-line-eq} with $v(0,\cdot;v_{0})=v_0$. On the one hand, by virtue of Proposition \ref{pro-Fk+}(iii)
 $$
 z(\psi)\geq2(m^{+}+1), \quad \forall\psi\in T_{u_{0}}W^{s}(\varphi_{+})\setminus\{0\}.
 $$
On the other hand, by Lemma \ref{Fk-iso-Ek}(i), $T_{u_{0}}W^{s}(\varphi_{+})$ is isomorphic to $E^+_{m^{+}+1}\oplus E^+_{m^{+}+2}\oplus\cdots$ ($E^+_{k},k\geq0$ are eigenspaces for the linearized equation of \eqref{r-d-eq} at $\varphi_+$). Therefore, $\operatorname{codim}T_{u_{0}}W^{s}(\varphi_{+})=2m^{+}+1.$

By Lemma \ref{def-Tu0}(ii), we have
$$T_{u_{0}}W^{u}(\varphi_{-})=F_{m^{-}+1}^{-}(0)=\{v_{0}\in X^{\alpha}|\liminf_{n\rightarrow-\infty}\|v(nT,\cdot;v_{0})\|^{\frac{1}{n}}_{X^{\alpha}}> r_{m^{-}+1}\},$$
where $F_{m^{-}+1}^{-}(0)$ is as defined in \eqref{def-Fk-} and $v(t,x;v_{0})$ is the backward continuation of $v_0$.
By Theorem \ref{index-cor} and Proposition \ref{pro-Fk-}(i), $m^{+}<m^{-}$; and moreover,
$$
F_{m^{-}+1}^{-}(0)\supset F_{m^{+}+1}^{-}(0)\quad \mbox{i.e.}\quad T_{u_{0}}W^{s}(\varphi_{-})\supset F_{m^{+}+1}^{-}(0).
$$
According to Proposition \ref{pro-Fk-}(iii),
$$
z(\psi)<2(m^{+}+1),\quad \forall\psi\in F_{m^{+}+1}^{-}(0)\setminus\{0\},
$$
and by Lemma \ref{Fk-iso-Ek}(ii), $F_{m^{+}+1}^{-}(0)$ is isomorphic to $E^-_{0}\oplus E^-_{1}\oplus\cdots\oplus E^{-}_{m^{+}}$ ($E^-_{k},k\geq0$ are eigenspaces for the linearized equation of \eqref{r-d-eq} at $\varphi_-$ ). Hence, $\dim F_{m^{+}+1}^{-}(0)=2m^{+}+1$.

Observing that $F_{m^{+}+1}^{-}(0)\cap F_{m^{+}+1}^{+}(0)=\{0\}$, we have that $T_{u_{0}}W^{s}(\varphi_{+})+T_{u_{0}}W^{u}(\varphi_{-})=X^{\alpha}.$
\end{proof}
\subsection{Proof of Theorem \ref{morse-smale-theor}}
\begin{proof}[Proof of Theorem \ref{morse-smale-theor}]
 Since we have already checked the map $P$ and the map $DP$ all are injective (see Lemma \ref{injective-map}) and proved the tranversality of stable and unstable manifolds for connecting hyperbolic points, the remainder is to show that the non-wandering set $\mathcal{N}$ of $P|_{\mathcal{A}}$ consists of finite number of hyperbolic periodic points of $P$. Let $\omega(u_0)$ be the $\omega$-limit set of $P$, then there exist some $w\in X^{\alpha}$ such that $\omega(u_0)\subset \{w_a|w_a=w(\cdot+a),\ a\in S^1\}$(see \cite[Theorem 1]{92San}), which combine with the fact that if  $\omega(u_0)$ is hyperbolic, it is then spatially-homogeneous (see \cite[Theorem 5.1]{19Shen}) imply $\omega(u_0)$ is a hyperbolic point.

 Note that hyperbolic points are discrete, and $\mathcal{A}$ is compact, then $\mathcal{N}\subset \mathcal{A}$ contains at most a finite number of hyperbolic points.
\end{proof}
\section{Appendix}
In this Appendix, we list some properties about bounded linear operators on Banach spaces and invariant manifolds theory for hyperbolic fixed points from \cite[Appendix B and C]{92Chen}.

Let $(X,\|\cdot\|)$ be a (real or complex) Banach space and $\mathcal{L}(X)$ be the Banach space of all bounded linear operators in $X$ equipped with the operator norm. In the following we are concerned with sequences $\{v(n)\}_{n\geq0}, v(n)\neq0 (n\geq0)$ and $\{\xi_{n}\}_{n\geq0}$ related by the formula
\begin{equation}\label{Sp and Rn-sequen}
v(n+1)=S_{p}v(n)+\xi_{n} \quad n\geq0,
\end{equation}
where $S_{p}\in \mathcal{L}(X)$. Sometimes, we require also
\begin{equation}\label{Rn-norm-limit}
\lim_{n\rightarrow\infty}\frac{\parallel \xi_{n}\parallel}{\parallel v(n)\parallel}=0.
\end{equation}

\begin{definition}
\begin{itemize}
  \item [\rm{(i)}]Denote by $\Lambda$ the set of all nonnegative numbers $\lambda$ for which $\sigma(S_{p})\cap\{z\in \mathbb{C}\mid|z|=\lambda\}\neq\emptyset$.
  \item [\rm{(ii)}]For each $a\in [0,\infty)\setminus\Lambda$, let $P(a)$ and $Q(a)$ be the projection operators associated with the decomposition of the spectrum $\sigma(S_{p})\cap\{|z|>a\}$ and $\sigma(S_{p})\cap\{|z|<a\}$, respectively, and let $U(a)$ and $V(a)$ be the restrictions of $S_{p}$ to $P(a)X$ and $Q(a)X$, respectively.
  \item [\rm{(iii)}]For each $a\in [0,\infty)\setminus\Lambda$, define
\begin{equation}
\mathcal{X}(a)=\lim_{n\rightarrow\infty}\frac{\parallel P(a)v(n)\parallel}{\parallel Q(a)v(n)\parallel}.
\end{equation}
\end{itemize}
\end{definition}

\begin{lemma}\label{projection-lem}
Let $\{v(n)\}_{n\geq0}\subset X$ satisfy  \eqref{Sp and Rn-sequen} and \eqref{Rn-norm-limit}
. If $[a,b]\cap\Lambda=\emptyset$, then one of the following alternative holds true:
\begin{itemize}
  \item [\rm{(i)}]$$ \lim_{n\rightarrow\infty}\frac{\parallel Pv(n)\parallel}{\parallel Qv(n)\parallel}=\infty \quad\mbox{and}\quad \liminf_{n\rightarrow\infty}\parallel v(n)\parallel^{1/n}\geq b;
$$
  \item [\rm{(ii)}]$$ \lim_{n\rightarrow\infty}\frac{\parallel Pv(n)\parallel}{\parallel Qv(n)\parallel}=0 \quad\mbox{and}\quad \limsup_{n\rightarrow\infty}\parallel v(n)\parallel^{1/n}\leq a;
$$
\end{itemize}
here $P:=P(a)=P(b),$ and $Q:=Q(a)=Q(b).$
\end{lemma}

\begin{lemma}\label{norm-spectrum}
Let $\{v(n)\}_{n\geq0}$ satisfy \eqref{Sp and Rn-sequen} and \eqref{Rn-norm-limit}. If $\Lambda$ is nowhere dense in $[0,\infty)$, then there exists $\lambda\in\Lambda$ such that
$$\lim_{n\rightarrow\infty}\parallel v(n)\parallel^{1/n}=\lambda.$$
\end{lemma}

\begin{lemma}\label{conver-posi-seque}
Let $\{v(n)\}_{n\geq0}$ satisfy \eqref{Sp and Rn-sequen} and \eqref{Rn-norm-limit}, and let $0<\alpha<\beta$ not be in $\Lambda$. Assume that $P_{*}:=P(\alpha)-P(\beta)$ has finite dimensional range, that is, $\dim P_{*}(X)<\infty$. If $\mathcal{X}(\alpha)=\infty$ and $\mathcal{X}(\beta)=0$, then the normalized sequence $\hat{v}(n):=v(n)/\parallel v(n)\parallel(n\geq0)$ is relatively compact in $X$, and its $\omega-$limit set
$$
\omega(\hat{v}):=\bigcap_{n\geq0}\mbox{closure}\{\hat{v}(m)|m\geq n\}
$$
is a compact set in $\{x\in P_{*}(X)|\parallel x\parallel=1\}$.
\end{lemma}

The lemma below is concerned with  the following iterations:
\begin{equation}\label{Sp and Rn-oper-sequen}
\begin{aligned}
v(n+1)&=S_{p}v(n)+R_{n}v(n)\quad n\geq0,\\
v(0)&=\psi\in X,
\end{aligned}
\end{equation}
where $S_{p}\in\mathcal{L}(X)$ and $\{R_{n}\}_{n\geq0}$ is a bounded sequence in $\mathcal{L}(X)$. For each $\psi\in X$, denoted by $v(n;\psi)$ the sequence defined by \eqref{Sp and Rn-oper-sequen}.

\begin{lemma}\label{posi-isomorphism}
Let $S_{p}\in \mathcal{L}(X)$ satisfy
\begin{equation}
(a,b)\cap\Lambda=\emptyset \quad\mbox{for some}\quad 0\leq a<b.
\end{equation}
Write $P=P(\mu), Q=Q(\mu),U=U(\mu)=S_{p}|_{P(\mu)X}$, and $V=V(\mu)=S_{p}|_{Q(\mu)X}$ for an arbitray $\mu\in(a,b).$ Let $\{R_{n}\}_{n\geq0}\subset\mathcal{L}(X)$ and define
\begin{equation}
\begin{aligned}
\delta(\lambda,R)&=\sum_{k=1}^{n}\lambda^{k-n-1}\parallel V^{n-k}QR_{k-1}\parallel\\
&+\sum_{k=n+1}^{\infty}\lambda^{k-n-1}\parallel U^{n-k}PR_{k-1}\parallel \quad a<\lambda<b.\nonumber
\end{aligned}
\end{equation}
Assume that
\begin{equation}\label{control-1}
\delta(\lambda,R)<1\quad\mbox{for some}\quad a<\lambda<b,
\end{equation}
and consider
\begin{equation}
F=\{\psi\in X|\sup_{n\geq0}\lambda^{-n}\parallel v(n;\psi)\parallel<\infty\}.
\end{equation}
Then, the projection $F\rightarrow Q(X):\psi\mapsto Q\psi$ gives an isomorphism between Banach spaces, $F$ is a closed subspace of $X$, and its codimension in $X$ is equal to that of $Q(X)$ in $X$.
\end{lemma}

\begin{lemma}\label{posi-norm-control}
Assume all the conditions in Lemma \ref{posi-isomorphism} are hold except \eqref{control-1}. Then, for any $\lambda\in(a,b),$ there exists a constant $M(\lambda)>0$ such that
$$
\delta(\lambda,R)\leq M(\lambda)\sup_{n\geq0}\parallel R_{n}\parallel.
$$
Therefore, choose $\lambda\in(a,b)$, then $\delta(\lambda,R)<1$ for $\sup_{n\geq0}\parallel R_{n}\parallel$  sufficiently small.
\end{lemma}

All of the results in the above have their counterparts for backward sequence  $\{v(n)\}_{n\leq0}, v(n)\neq0$,
\begin{equation}\label{nega-Sp and Rn-sequen}
\begin{aligned}
v(n+1)&=S_{p}v(n)+\xi_{n} \quad n<0,\\
v(0)&=\psi\in X,
\end{aligned}
\end{equation}
where $S_{p}\in \mathcal{L}(X), v(n)\neq0(n\leq0)$, and $\{\xi_{n}\}_{n<0}\subset X$, satisfies
\begin{equation}\label{nega-norm-Rn}
\lim_{n\rightarrow-\infty}\frac{\parallel\xi_{n}\parallel}{\parallel v(n)\parallel}=0.
\end{equation}

\begin{lemma}\label{conver-nega-seque}
Let $\{v(n)\}_{n\leq0}\subset X$ satisfy \eqref{nega-Sp and Rn-sequen} and \eqref{nega-norm-Rn}. Then the statements in Lemma \ref{projection-lem}, Lemma \ref{norm-spectrum} and Lemma \ref{conver-posi-seque} remain true under the corresponding conditions, if all of the limits as $n\rightarrow\infty$ there are replaced by the corresponding limits as $n\rightarrow-\infty$ and if the $\omega-$limit set in Lemma \ref{conver-posi-seque} is replaced by $\alpha-$limit set defined by
$$
\alpha(\hat{v}):=\bigcap_{n\leq0}\mbox{closure}\{\hat{v}(m)|m\leq n\}
$$
where the closure is with respect to the strong topology of $X$.
\end{lemma}

We also consider the backward analogue of \eqref{Sp and Rn-oper-sequen},
\begin{equation}\label{nega-Sp and Rn-oper-sequen}
\begin{aligned}
v(n+1)&=S_{p}v(n)+R_{n}v(n)\quad n<0,\\
v(0)&=\psi\in X,
\end{aligned}
\end{equation}
where $S_{p}\in\mathcal{L}(X)$ and $\{R_{n}\}_{n<0}\subset\mathcal{L}(X)$.
\begin{lemma}\label{nega-isomorphism}
Let $S_{p}\in \mathcal{L}(X)$ satisfy
\begin{equation}
(a,b)\cap\Lambda=\emptyset \quad\mbox{for some}\quad 0\leq a<b.
\end{equation}
Write $P=P(\mu), Q=Q(\mu),U=U(\mu)$, and $V=V(\mu)$ for an arbitrary $\mu\in(a,b).$ Let $\{R_{n}\}_{n\leq0}\subset\mathcal{L}(X)$ and define
\begin{equation}
\begin{aligned}
\delta(\lambda,R)&=\sum_{n+1\leq k\leq0}\lambda^{k-n-1}\parallel U^{n-k}PR_{k-1}\parallel\\
&+\sum_{k\leq n}\lambda^{k-n-1}\parallel V^{n-k}QR_{k-1}\parallel.\nonumber
\end{aligned}
\end{equation}
Assume that
\begin{equation}\label{nega-control-1}
\delta(\lambda,R)<1\quad\mbox{for some}\quad a<\lambda<b,
\end{equation}
and consider
\begin{equation}
F=\{\psi\in X|\mbox{there exists} \{v(n)\}_{n\leq0}\mbox{satisfying \eqref{nega-Sp and Rn-oper-sequen}}and\sup_{n\leq0}\lambda^{-n}\parallel v(n)\parallel<\infty\}.
\end{equation}
Then,
\begin{itemize}
  \item [\rm{(i)}]the projection $F\rightarrow P(X):\psi\mapsto P\psi$ is an isomorphism and $\dim F=\dim PX$;
  \item [\rm{(ii)}]for each $\psi\in F$, the sequence $\{v(n)\}_{n\leq0}$ satisfying \eqref{nega-Sp and Rn-oper-sequen} and
$$
\sup_{n\leq0}\lambda^{-n}\parallel v(n)\parallel<\infty
$$
is unique.
\end{itemize}
\end{lemma}

\begin{lemma}\label{nega-norm-control}
Assume all the conditions in Lemma \ref{nega-isomorphism} are hold except \eqref{nega-control-1}. Then, for any $\lambda\in(a,b),$ there exists a constant $M(\lambda)>0$ such that
$$
\delta(\lambda,R)\leq M(\lambda)\sup_{n<0}\parallel R_{n}\parallel.
$$
where $\delta(\lambda,R)$ is as in Lemma \ref{nega-isomorphism}. Therefore, choose $\lambda\in(a,b)$, then $\delta(\lambda,R)<1$ for $\sup_{n\leq 0}\parallel R_{n}\parallel$  sufficiently small.
\end{lemma}

Assume that
\begin{itemize}
  \item
      [\rm{(S.1)}]$\mathbb{U}\subset X$ is an open set;
  \item [\rm{(S.2)}]$P:\mathbb{U}\rightarrow X$ is a $C^{1}$ map.
\end{itemize}

The following two lemmas concerned with invariant manifold theory.
\begin{lemma}\label{unstable-manif}
Assume (S.1)-(S.2) and let $\varphi\in\mathbb{U}$ be a fixed point of $P$. Suppose that $\sigma(DP(\varphi))=\sigma_{1}\cup\sigma_{2}$ with $\sigma_{1}=\{z\in\mathbb{C}\mid|z|\geq\alpha_{1}\}$ and $\sigma_{2}=\{z\in\mathbb{C}\mid|z|\leq\alpha_{2}\}$ for some $0\leq \alpha_{2}<1<\alpha_{1}\leq \infty$. Decompose $X=X_{1}\oplus X_{2}$ to $DP(\varphi)$-invariant subsequences of $X$ corresponding to spectral sets $\sigma_{1}$ and $\sigma_{2}$. Then there exist a neighborhood $V_{1}$ of 0 in $X_{1}$, a neighborhood $V$ of 0 in $X$, and a $C^{1}$ map $g:V_{1}\rightarrow X_{2}$ such that the $C^{1}$ submanifold of $X$
\begin{equation}
W_{loc}^{u}=W_{loc}^{u}(\varphi):=\{\varphi+a+g(a)\mid a\in V_{1}\}
\end{equation}
satisfies the following \rm{(i)-(v)}:
\begin{itemize}
  \item [\rm{(i)}]for any $x\in W_{loc}^{u}$, there is a negative semiorbit $\{x_{n}\}_{n\leq0}$ through $x_{0}=x$ such that $x_{n}-\varphi\in V, x_{n}\in W_{loc}^{u}(n\leq0)$ and that
\begin{equation}\label{xn-sequen}
\limsup_{n\rightarrow-\infty}\parallel x_{n}-\varphi\parallel^{1/|n|}\leq\alpha_{1}^{-1};
\end{equation}
  \item [\rm{(ii)}]if $\{x_{n}\}_{n\leq0}$ is a negative semiorbit with $x_{n}-\varphi\in V\cap(V_{1}+X_{2})(n\leq0)$, then $\{x_{n}\}_{n\leq0}\subset W_{loc}^{u}$ and \eqref{xn-sequen} holds true;
  \item [\rm{(iii)}]$X_{1}$ is the tangent space of $W^{u}_{loc}$ at $\varphi: T_{\varphi}W^{u}_{loc}=X_{1}$;
  \item [\rm{(iv)}]if $\{x_{n}\}_{n\leq0}\subset W^{u}_{loc}$ is a negative semiorbit and if $y_{0}\in T_{x_{0}}W^{u}_{loc}$, then there exist $\{y_{n}\}_{n\leq0}\subset X$ such that $y_{n+1}=DP(x_{n})y_{n} (n<0), y_{n}\in T_{x_{n}}W^{u}_{loc}(n<0)$ and
\begin{equation}\label{yn-sequen}
\limsup_{n\rightarrow-\infty}\parallel y_{n}\parallel^{1/|n|}\leq\alpha_{1}^{-1};
\end{equation}
  \item [\rm{(v)}] if $\{x_{n}\}_{n\leq0}\subset W^{u}_{loc}$ is a negative semiorbit and $\{y_{n}\}_{n\leq0}\subset X$ is such that $y_{n+1}=DP(x_{n})y_{n}(n<0)$ and
\begin{equation}
\limsup_{n\rightarrow-\infty}\parallel y_{n}\parallel^{1/|n|}<\alpha_{2}^{-1},
\end{equation}
then $y_{n}\in T_{x_{n}}W_{loc}^{u}(n\leq0)$ and \eqref{yn-sequen} holds true.
\end{itemize}
\end{lemma}

\begin{lemma}\label{stable-manif}
Assume conditions in Lemma \ref{unstable-manif} are hold. Then there exist a neighborhood $V_{2}$ of 0 in $X_{2}$, a neighborhood $V$ of 0 in $X$, and a $C^{1}$ map $h:V_{2}\rightarrow X_{1}$ such that the $C^{1}$ submanifold
\begin{equation}
W_{loc}^{s}=W_{loc}^{s}(\varphi):=\{\varphi+b+h(b)\mid b\in V_{2}\}
\end{equation}
satisfies the following \rm{(i)-(v)}:
\begin{itemize}
  \item [\rm{(i)}]if $x\in W_{loc}^{s}$, then $P^{n}x\in W_{loc}^{s}$ and $P^{n}x-\varphi\in V(n\geq0)$, and
\begin{equation}\label{Pn-sequen}
\limsup_{n\rightarrow\infty}\parallel P^{n}x-\varphi\parallel^{1/n}\leq\alpha_{2};
\end{equation}
  \item [\rm{(ii)}]if $\{x_{n}\}_{n\geq0}$ is a positive semiorbit with $x_{n}-\varphi\in V\cap(V_{2}+X_{1})(n\geq0)$, then $\{x_{n}\}\in W_{loc}^{s}(n\geq0)$ and \eqref{Pn-sequen} holds true;
  \item [\rm{(iii)}]$T_{\varphi}W^{s}_{loc}=X_{2}$;
  \item [\rm{(iv)}]if $\{x_{n}\}_{n\geq0}\subset W^{s}_{loc}$ is a positive semiorbit and  $y_{0}\in T_{x_{0}}W^{s}_{loc}$, then the sequence $\{y_{n}\}_{n\geq0}$ defined by $y_{n+1}=DP(x_{n})y_{n} (n\geq0)$ satisfies that $y_{n}\in T_{x_{n}}W^{s}_{loc}$ and
\begin{equation}\label{stable-yn-sequen}
\limsup_{n\rightarrow\infty}\parallel y_{n}\parallel^{1/n}\leq\alpha_{2};
\end{equation}
  \item [\rm{(v)}]if $\{x_{n}\}_{n\geq0}\subset W^{s}_{loc}$ is a positive semiorbit and $\{y_{n}\}_{n\geq0}$ is such that $y_{n+1}=DP(x_{n})y_{n}(n\geq0)$ and
\begin{equation}
\limsup_{n\rightarrow\infty}\parallel y_{n}\parallel^{1/n}<\alpha_{1},
\end{equation}
then $y_{n}\in T_{x_{n}}W_{loc}^{s}(n\geq0)$ and \eqref{stable-yn-sequen} holds true.
\end{itemize}
\end{lemma}

\section*{Acknowledgments}
Dun Zhou was  partially supported by National Science Foundation of China under grants (No. 11971232, 12071217).

\end{document}